\documentclass[11pt]{aart}

\usepackage[letterpaper, hmargin=1in, top=1.1in, bottom=1.3in, footskip=0.7in]{geometry}

\usepackage{titlesec}
\titleformat{\subsection}[runin]{\normalfont\bfseries}{\thesubsection.}{.5em}{}[.]\titlespacing{\subsection}{0pt}{2ex plus .1ex minus .2ex}{.8em}
\titleformat{\subsubsection}[runin]{\normalfont\itshape}{\thesubsubsection.}{.3em}{}[.]\titlespacing{\subsubsection}{0pt}{1ex plus .1ex minus .2ex}{.5em}
\titleformat{\paragraph}[runin]{\normalfont\itshape}{\theparagraph.}{.3em}{}[.]\titlespacing{\paragraph}{0pt}{1ex plus .1ex minus .2ex}{.5em}

\usepackage{amsmath,bbm,color,amssymb,subfigure,graphicx,amsthm}

\usepackage{cite}

\definecolor{vdarkred}{rgb}{0.6,0,0.2}
\definecolor{vdarkblue}{rgb}{0,0.2,0.6}
\usepackage[pdftex, colorlinks, linkcolor=vdarkblue,citecolor=vdarkred]{hyperref}
\usepackage[small]{caption}

\newcommand{\ld}{\ldots}
\newcommand{\beg}{\begin}
\newcommand{\en}{\end}

\newcommand{\trm}{\textrm}
 
\newcommand{\bgt}{\begin{itemize}}
\newcommand{\ent}{\end{itemize}}
\newcommand{\ite}{\item}

\newcommand{\eqre}{\eqref}
\newcommand{\la}{\label}

\newcommand{\rfl}{\rfloor}
\newcommand{\lfl}{\lfloor}

\newcommand{\Var}{\operatorname{Var}}

\newcommand{\Cov}{\operatorname{Cov}}

 \newcommand{\bgn}{\begin{enumerate}}
\newcommand{\enn}{\end{enumerate}}
\newcommand{\ds}{\displaystyle}
 
\newcommand{\p}{\mathbb{P}}

\newcommand{\E}{\mathbb{E}}

\newcommand{\whp}{w.h.p.\ }
\newcommand{\whpv}{w.h.p., }

\newcommand{\pro}{probability }

\newcommand{\f}{\frac}
\newcommand{\ff}{\frac{1}}
\newcommand{\lf}{\left}
\newcommand{\ri}{\right}
\newcommand{\st}{such that }

\newcommand{\lam}{\lambda}

\newcommand{\vfi}{\varphi}
\newcommand{\ste}{\, :\, }
\newcommand{\mc}{\mathcal }

\newcommand{\eps}{\varepsilon}

\newcommand{\cN}{\mathcal{N}}

\newcommand{\bck}{\backslash}

\newcommand{\udl}{\underline}
\newcommand{\bbm}{\begin{bmatrix}}
\newcommand{\ebm}{\end{bmatrix}}
\newcommand{\bes}{\begin{equation*}}
\newcommand{\ees}{\end{equation*}}
\newcommand{\be}{\begin{equation}}
\newcommand{\ee}{\end{equation}}
\newcommand{\one}{\mathbbm{1}}

\newcommand{\ie}{i.e.\ }

\newcommand{\bpm}{\begin{pmatrix}}
\newcommand{\epm}{\end{pmatrix}}

\newcommand{\cd}{\cdots}

\newcommand{\bpr}{\beg{proof}}
\newcommand{\epr}{\en{proof}}

\newcommand{\del}{\delta}
\newcommand{\Del}{\Delta}
\newcommand{\da}{\downarrow}

\newcommand{\uA}{\udl{A}}

\newcommand{\ka}{\kappa}

\newcommand{\pma}{p_{\max}}
\newcommand{\qma}{q_{\max}}

\newcommand*{\deq}{\mathrel{\vcenter{\baselineskip0.65ex \lineskiplimit0pt \hbox{.}\hbox{.}}}=}
\newcommand*{\eqd}{=\mathrel{\vcenter{\baselineskip0.65ex \lineskiplimit0pt \hbox{.}\hbox{.}}}}
\newcommand{\ABS}[1]{{{\left|#1\right|}}} 
\newcommand{\PAR}[1]{{{\left(#1\right)}}} 
\newcommand{\AND}{{\quad \trm{ and } \quad }}

\newcommand{\BIN}{\mathrm{Bin}}

\newcommand{\BER}{\mathrm{B}}

\newtheorem{Th}{Theorem}[section]

\newtheorem{lem}[Th]{Lemma}

\newtheorem{cor}[Th]{Corollary}

\theoremstyle{definition}

\newtheorem{rmk}[Th]{Remark}

\newtheorem{Def}[Th]{Definition}

\definecolor{darkblue}{rgb}{0,0.3,0.9}

\renewcommand{\cal}{\mathcal}

\newcommand{\ul}[1]{\underline{#1} \!\,} 

\newcommand{\me}{\mathrm{e}}

\newcommand{\col}{\mathrel{\vcenter{\baselineskip0.75ex \lineskiplimit0pt \hbox{.}\hbox{.}}}}

\renewcommand{\leq}{\leqslant}
\renewcommand{\geq}{\geqslant}
\renewcommand{\le}{\leqslant}
\renewcommand{\ge}{\geqslant}
\renewcommand{\epsilon}{\varepsilon}

\newcommand{\floor}[1] {\lfloor {#1} \rfloor}

\newcommand{\pb}[1]{\bigl({#1}\bigr)}
\newcommand{\pB}[1]{\Bigl({#1}\Bigr)}
\newcommand{\pbb}[1]{\biggl({#1}\biggr)}

\newcommand{\qb}[1]{\bigl[{#1}\bigr]}

\newcommand{\hbb}[1]{\biggl\{{#1}\biggr\}}

\newcommand{\absB}[1]{\Bigl\lvert #1 \Bigr\rvert}
\newcommand{\absbb}[1]{\biggl\lvert #1 \biggr\rvert}

\numberwithin{equation}{section}

\allowdisplaybreaks
 
\title{Largest eigenvalues of sparse inhomogeneous Erd\H{o}s-R\'enyi graphs}
\author{Florent Benaych-Georges \and Charles Bordenave   \and Antti Knowles}

\begin{document}
\maketitle

\beg{abstract}We consider inhomogeneous Erd\H{o}s-R\'enyi graphs. We suppose that the maximal mean degree $d$ satisfies $d \ll \log n$. We characterize the asymptotic behavior of the $n^{1 - o(1)}$ largest  eigenvalues of the adjacency matrix and its centred version. We prove that these extreme eigenvalues are governed at first order by the largest degrees and, for the adjacency matrix, by the nonzero eigenvalues of the expectation matrix. Our results show that the extreme eigenvalues exhibit a novel behaviour which in particular rules out their convergence to a nondegenerate point process.
Together with the companion paper \cite{FBGCBAK2016}, where we analyse the extreme eigenvalues in the complementary regime $d \gg \log n$, this establishes a crossover in the behaviour of the extreme eigenvalues around $d \sim \log n$.
Our proof relies on a new tail estimate for the Poisson approximation of an inhomogeneous sum of independent Bernoulli random variables, as well as on an estimate on the operator norm of a pruned graph due to Le, Levina, and Vershynin from \cite{LeVershynin}.
\en{abstract}

\section{Introduction} 
 
The purpose of the present text is to understand the extreme eigenvalues of the adjacency matrix of an inhomogeneous Erd\H{o}s-R\'enyi random graph on $n$ vertices in the regime where the maximal mean degree $d$ satisfies $d \ll \log n$. Heuristically, such eigenvalues arise from three different origins: (i) the edge of the limiting bulk eigenvalue density, (ii) vertices of large degrees, and (iii) outliers associated with nonzero eigenvalues of the expectation matrix. One goal of this paper is a precise understanding of this interplay between random matrices on the one hand and the geometry of random graphs on the other. Such questions have several motivations from applications, such as the estimation of the spectral gap and spectral clustering.

The simplest random graph is the Erd\H{o}s-R\'enyi random graph $G(n,d/n)$, where each edge is present independently with probability $d/n$. In this case it is rather well understood that the behaviour of the extreme eigenvalues in the regime $d \gg \log n$ is governed by random matrix behavior; see \cite{KF,VuLargestEig,MR2155709,FBGCBAK2016, LS1, EKYY1,EKYY2}. In the complementary regime $d \ll \log n$, the main result available up to now was due to Sudakov and Krivelevich \cite{MR1967486}, who showed that the largest eigenvalue of the adjacency matrix is asymptotically equivalent to the maximum of the maximal mean degree $d$ and the square root of the largest degree (their result holds in fact for all regimes of $d$). 

Our main result is a description of the behaviour of the $n^{1 - o(1)}$ largest and smallest eigenvalues of the adjacency matrix $A$ and its centred version $\ul A \deq A - \E A$, for an inhomogeneous Erd\H{o}s-R\'enyi random graph whose mean degree $d$ is much smaller than $\log n$. Informally, we prove that the $k$-th largest eigenvalue eigenvalue of $\ul A$ satisfies
\begin{equation} \label{lambda_intro}
\lambda_k(\ul A) \approx \sqrt{\frac{(\log ( n / k) }{\log ((\log n) / d)}}\,, \qquad k \geq n^{1-\eps} \,, \qquad \eps \in (0,1)\,.
\end{equation}
Under mild additional assumptions (satisfied for instance by stochastic block models), we show that the same result holds for the eigenvalues of $A$, with the exception of some outlier eigenvalues whose locations we also characterize.

A consequence of our results, combined with those from the companion paper \cite{FBGCBAK2016}, where we analyse the extreme eigenvalues in the complementary regime $d \gg \log n$, is a crossover in the behaviour of the extreme eigenvalues around $d \sim \log n$ (the same threshold as for the graph connectivity). Indeed, in \cite{FBGCBAK2016} we prove that if $d \gg \log n$ then all eigenvalues are asymptotically contained within the support of the semicircle law describing the macroscopic eigenvalue density, while in the current paper we establish for $d \ll \log n$ a novel behaviour of the extreme eigenvalues, which implies that $n^{1 - o(1)}$ eigenvalues escape the support of the semicircle law. Their locations are governed by \eqref{lambda_intro} and define a distribution that is illustrated in Figure \ref{Fig:histo_edge} below.

It is helpful to analyse the behaviour of the extreme eigenvalues for $d \ll \log n$ in the context of random matrix theory.
Until now, in random matrix theory two different types of universal behaviour at leading order of the extreme eigenvalues have been established, exhibited for instance by light- and  heavy-tailed Wigner matrices respectively. After a suitable deterministic rescaling of the matrix, these two classes may be characterized as follows.
\begin{itemize}
\item[(a)]
The extreme eigenvalues converge to the edge of the support of the asymptotic bulk spectrum.
\item[(b)]
The extreme eigenvalues form asymptotically a Poisson point process.
\end{itemize}
For example, it is known \cite{ABP,SoshPoi,YL} that a Wigner matrix whose entries have tail decay $x^{-\alpha}$ belongs to class (a) if $\alpha > 4$ and to class (b) if $\alpha < 4$. Moreover, as stated above, in the companion paper  \cite{FBGCBAK2016} we prove that the Erd\H{o}s-R\'enyi graph belongs to class (a) if $d \gg \log n$. Also, sparse heavy-tailed random matrices exhibit a transition between these classes depending on the sparsity and the tail decay of the entries \cite{BGPech}.

A consequence of our results is that, perhaps surprisingly, for $d \ll \log n$, the (possibly inhomogeneous) Erd\H{o}s-R\'enyi graph belongs to neither class (a) nor class (b). Instead, the behaviour from \eqref{lambda_intro} results in a sharp increase in the density of eigenvalues as one moves towards the centre of the spectrum, which implies that, no matter the rescaling of the spectrum, any nondegenerate limiting point process will be infinite on compact sets.

The proof consists of two main steps. In a first step, we analyse the distribution of the $n^{1 - o(1)}$ largest degrees, and prove that the corresponding vertices are with high probability separated by distance at least 3 from each other. The key tool behind this step is a new sharp estimate (Theorem \ref{th:tail} below) on the tail of a sum of inhomogeneous independent Bernoulli random variables. This estimate may be regarded as an improvement for the tails of the well-known Poisson approximation provided by Le Cam's inequality \cite{BHJBook}. It is of independent interest. In a second step, we compare the $n^{1 - o(1)}$ largest eigenvalues of the graph with those of the graph obtained by only keeping the edges incident to the $n^{1 - o(1)}$ vertices of largest degree. The latter corresponds to a block-diagonal matrix whose blocks are associated with star graphs of high central degree. This comparison is based on a sharp estimate on the operator norm of the complementary graph due to Le, Levina, and Vershynin \cite[Theorem 2.1]{LeVershynin}.

This text is organized as follows. In the remainder of the introduction, we state our main results, which are proved in Section 
\ref{sec:proof}. In Section \ref{sec:Poisson} we state and prove the new tail estimate for Poisson approximation mentioned above.

\paragraph{Notation} 
   The eigenvalues of a Hermitian $n \times n$ matrix $H$ are denoted by 
 $\lam_1(H)\ge \lam_2(H)\ge\cd \geq \lambda_n(H)$. Its operator norm is given by $\| H \|  =  \max( \lambda_1 (H) , \lambda_1 (-H) ) $.
For $p\in [0,1]$, we denote by $\BER(p)$ the Bernoulli law with parameter $p$, \ie $\BER(p)=(1-p)\del_0+p\del_1$. We denote by $\BIN( p_1, \ldots , p_n)$ the law $\BER(p_1) * \cdots * \BER(p_n)$.
In particular, $\BIN(p,\ldots, p)$ is the Binomial distribution with parameters $(n,p)$. For $x > 0$ use the abbreviation $[x] \deq \{1,2, \dots, \floor{x}\}$.

 \subsection{Hypotheses and definitions}

Throughout this paper, $A$ is the  adjacency matrix of an inhomogeneous (undirected) Erd\H{o}s-R\'enyi random graph $G$ with vertex set $[n]$, where the edge $\{i,j\}$ is included with \pro $p_{ij} \in [0,1]$ independently of the others. Note that we allow loops: there is a loop at vertex $i$ with probability $p_{ii}$.

The maximal edge probability is
$$
\pma \deq \max_{i \ne j}p_{ij}\,.  
$$
The mean degree of the vertex $i \in [n]$ and the maximal mean degree are defined as
$$
d_i \deq \sum_j p_{ij} \,, \qquad    d \deq \max_i d_i
   $$
respectively.

We always suppose that there are  $\ka>0$ and $\eta\in (0,1)$ \st \be\la{eq:ordersA}\ka\le d\le \eta\log n\qquad\AND\qquad \pma\le n^{-1+\eta}\,.
\ee
As all of our error term controls will be uniform, with quantitative rates of convergence, in the parameters $(p_{ij})_{i,j\in [n]}$ \st \eqref{eq:ordersA} holds, we introduce the following definitions.

 \beg{Def}\la{def:error_control}\begin{enumerate}
 \ite An \emph{admissible error function} is a function $\psi(n,\ka,\eta)$ satisfying \be\la{def:03041718h49}\forall\ka>0\,, \qquad \lim_{\substack{n\to\infty\\ \eta\to 0}}\psi(n,\ka,\eta)=0\,.\ee
 \ite Given an event $E$ and a condition $\cal A$ on the parameters $\kappa,\eta,(p_{ij})_{i,j\in [n]}$, we say that, \emph{under $\cal A$, $E$ holds with high probability (w.h.p.)}\ if there is an admissible error function $\psi(n,\ka,\eta)$  
 such that  $$\p(E)\ge 1-\psi(n,\ka,\eta)$$ for all $\kappa,\eta,(p_{ij})_{i,j\in [n]}$ satisfying $\cal A$.
 \ite Given a condition $\cal A$ on the parameters $\kappa,\eta,(p_{ij})_{i,j\in [n]}$, for two families of random variables $(u_t), (v_t)$ we say that \emph{under $\cal A$, for all $t$, $u_t \sim v_t$} if there is an admissible error function $\psi(n,\ka,\eta)$ 
 such that  $$\p\lf(\absbb{\f{v_t}{u_t}-1}\le \psi(n,\ka,\eta) \text{ for all } t \ri)\ge 1-\psi(n,\ka,\eta)$$ for all $\kappa,\eta,(p_{ij})_{i,j\in [n]}$ satisfying $\cal A$.
 \end{enumerate} 
 \end{Def}
 Let us emphasize that the point in this definition is the uniformity of the error terms in the asymptotic regime where $n\to \infty$, $d=o( \log n)$, and $\pma=n^{-1+o(1)}$.
To simplify presentation, in the following we shall not identify the error functions $\psi(n,\kappa,\eta)$ explicitly, although a careful look at our proofs will easily yield explicit expressions for them.

Finally, for $k \in [n]$ we set
\begin{equation}\label{eq:defLk}
L_k \deq \f{\log (n/k)}{\log((\log n)/d)}\,.
\end{equation}

 \subsection{Relation between the centred adjacency matrix and the largest degrees}

For $i \in [n]$, let $D_i$ denote the degree of the vertex $i$ in the graph $G$.
Denote by $$D_1^{\da}\ge \cd\ge D_n^{\da}$$ the decreasingly ordered degrees $D_1, \dots, D_n$. We also introduce the centred adjacency matrix
 $$
 \uA \deq A - \E A\,. 
 $$
By definition, $(\E A)_{ij} = p_{ij}$. 
 The following theorem relates the largest eigenvalues of $\uA$ to the largest degrees, whose behaviour is described in Propositions \ref{propo_position_degree_max} and \ref{27516} below.
 \beg{Th}\la{thalphall12}For any $\eps  \in (0, 1)$, under \eqref{eq:ordersA}, w.h.p.,
\begin{align}\la{Estlam816} \max_{k \in [n^{1-\eps}]} \absB{\lam_k (\uA)-\sqrt{D_k ^{\da}}} &\le C  \sqrt{ n\pma }  + \eps\sqrt{L_1} \,,\\
\la{Estlam8160}\max_{k \in [n^{1-\eps}]} \absB{\lam_{n+1-k} (\uA)+\sqrt{D_k ^{\da}}} &\le  C \sqrt{ n\pma }  +  \eps \sqrt{L_1} \,,
\end{align}  
   where $C$ is a universal constant and $L_1$ is defined in \eqref{eq:defLk}. 
\en{Th}

The proof of Theorem \ref{thalphall12} is based on an analysis of the graph spanned by the largest degree vertices, and on \cite[Theorem 2.1]{LeVershynin} due to Le, Levina, and Vershynin on the operator norm of the centred adjacency matrix where all large degree vertices have been removed. The term $\pm \sqrt{D_k ^{\da}}$ arises as an eigenvalue of a star graph with central degree $D_k^{\da}  - O(1)$ (see Definition \ref{def:star_graph} below). 

By Proposition \ref{propo_position_degree_max} below, for any $\eps \in (0,1)$, we have $D_1^\da \leq (1+\eps) L_1$ w.h.p., which yields the following corollary.
\begin{cor}\label{cor:barA0}
For any $\eps \in (0,1)$, under \eqref{eq:ordersA}, w.h.p.,
$$
\|\uA \| \leq (1 + \eps) \sqrt L_1 + C \sqrt{ n \pma}\,.
$$
\end{cor}

As explained, for example, in \cite{LeVershynin}, Corollary  \ref{cor:barA0} finds applications in the analysis of spectral clustering techniques on random graphs. 

Under the additional hypothesis that all vertices have the same mean degree, the behaviour of the largest degrees summarized in Corollary \ref{cor:equivd} below implies that $D_k^{\da} \sim L_k$, where $L_k$ was defined in \eqref{eq:defLk}.
We deduce the following result.

\begin{cor}\label{cor:barA} Let $\eps \in (0,1)$. Then under the conditions $d_i = d$ for all $i$, $ n\pma \le \eta L_1$, and \eqref{eq:ordersA}, we have for all $k \in [n^{1-\eps}]$
\be\la{EstLam916FB}\lam_k(\uA)\sim \sqrt{ L_k} \AND \lam_{n+1-k} (\uA)\sim  - \sqrt{ L_k}\,.\ee
\end{cor}

\beg{rmk}\la{rem:pospart2}
There is an equivalent way to state Corollary \ref{cor:barA}. Introduce the counting function of the renormalized eigenvalues of $\uA=A-\E A$, defined as \be\la{eq:defNAu}N_{\uA}(x) \deq \#\hbb{k \in [n] \ste \frac{\lam_k (\uA)}{ \sqrt L_1} \ge x}\,.\ee
The first estimate of \eqref{EstLam916FB} implies  that for any $x \in (0, 1)$,    \be\la{EstN816} \f{\log N_{\uA}(x)}{\log n}\sim 1-x^2\,.\ee
Indeed, for any $\del>0$ small enough, for $k:=\lfl n^{1-x^2-\del}\rfl$ and $k':=\lceil n^{1-x^2+\del}\rceil$. We have 
$$k\le N_{\uA}(x)<k'\iff \lam_{k'}(\uA)<x\sqrt{L_1}\le \lam_k\,,$$ 
which happens w.h.p.\ by the first estimate of \eqref{EstLam916FB}.

Informally, \eqref{EstN816} states that $N_{\uA}(x)\approx n^{1-x^2}$, from which we deduce that the density of renormalized eigenvalues 
$ \frac{\lam_k (\uA)}{ \sqrt L_1}$ at $x\in (0,1)$ is asymptotically
\begin{equation} \la{dens}
2 \log (n) \, n^{1 - x^2} x\,.
\end{equation}
See Figure \ref{Fig:histo_edge} below for an illustration.
\en{rmk}

\begin{figure}[!ht]
\centering
\includegraphics[scale=.4]{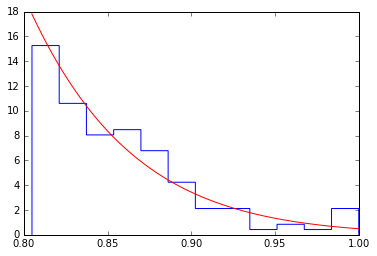}\includegraphics[scale=.4]{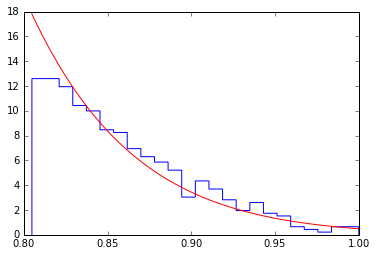}\includegraphics[scale=.4]{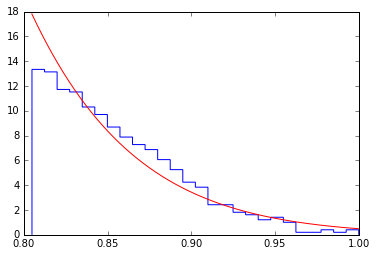}
\caption{Histogram of the right edge of the spectrum of $A$ in the case of a homogenous Erd\H{o}s-R\'enyi graph and density of \eqre{dens}. The eigenvalues are renormalized in such a way that $\lam_2=1$ ($\lam_1$ has been removed) and the histogram as well as the density are normalized in such a way that the total area is $1$. Here, $n=5\cdot 10^4$ and $d=0.5$ (left), $d=1.5$ (centre), $d=2.5$ (right). We see that as $d$ grows, the empirical density gets more convex at its edge, which agrees with the idea that the semicircle law approximation  gets more accurate.}\la{Fig:histo_edge}
\end{figure}

\begin{rmk} \label{rmk:Xi}
The estimate \eqref{EstN816} states there exists no deterministic sequence $\alpha = \alpha_n$ such that the point process
\begin{equation*}
\Xi \deq \{\alpha \lambda_k(\ul A) \col k \in [n]\}
\end{equation*}
is asymptotically finite and nonzero on compact sets. In particular, $\Xi$ cannot converge to a point process as $n \to \infty$. Note, however, that our results do not rule out the existence of an affine transformation parametrized by $\alpha = \alpha_n$ and $\beta = \beta_n$ such that the point process $\{\alpha (\lambda_k(\ul A) - \beta) \col k \in [n]\}$ converges.
\end{rmk}

  \subsection{Consequences for the adjacency matrix}

Gershgorin's Circle Theorem implies that
$$
\| \E A \| \leq d\,. 
$$ Then, writing $A=\uA+\E A$,  following corollary is an immediate consequence of Theorem \ref{thalphall12} and Weyl's inequality (see e.g.\ \cite[Corollary III.2.6]{Bhatia}).

  \beg{cor}\la{thalphall12cor}
  \bgt\ite[(a)]
  Theorem \ref{thalphall12} holds with $\ul A$ replaced by $A$ and the right-hand sides of \eqref{Estlam816}--\eqref{Estlam8160} replaced by $\ds C \sqrt{n\pma } + \eps\sqrt{L_1} +  d$.
\ite[(b)] Under \eqref{eq:ordersA}, w.h.p., for any $k \in [n]$,
 $$
\ABS{ \lambda_k (A)  -  \lambda_k ( \E A) } \leq C \pB{ \sqrt L_1 + \sqrt {n \pma}}\,,
 $$
 for some universal constant $C$.
 \ent
\en{cor}

\beg{rmk}
As $L_1\sim D_1^\da$, for an homogenous Erd\H{o}s-R\'enyi random graph, Corollary \ref{thalphall12cor}    is consistent with \cite[Theorem 1.1]{MR1967486} which asserts that $\lambda_1 (A)  = \| A \| \sim \max \{ \sqrt{D_1^\da} , d\}$ in all regimes of $d$.
\en{rmk}

\subsection{Applications to stochastic block models}
In the stochastic block model, $\E A$ has bounded rank and all its nonzero eigenvalues are of order $d$. We denote by $\lambda_1^+ \geq \cdots \geq \lambda_{k^+}^+ > 0$ the positive eigenvalues of $\E$ and $\lambda_1^- \leq \cdots \leq \lambda_{k^-}^- < 0$ the negative eigenvalues of $\E A$. For $\ka$ the constant of   \eqref{eq:ordersA}, we suppose that \be\la{541714h} d\ge \lambda_1^+ \geq \cdots \geq \lambda_{k^+}^+ \ge\ka d\,,\quad -d \leq \lambda_1^- \leq \cdots \leq \lambda_{k^-}^-\leq -\kappa d\,,\quad k^+ +k^- \le \ka^{-1}\,.
\ee 
Then, there is a dichotomy in the behaviour of the $k^+ + k^-$ largest (in absolute value) eigenvalues of $A$, depending on whether $d \ll \sqrt L_1$ or $d \gg \sqrt L_1$. Under mild conditions on $d$, these conditions read $d \ll \sqrt { \log n / \log \log n}$ or $d \gg \sqrt { \log n / \log \log n}$.

\beg{propo}\la{propo:SBM}Let $\eps \in (0,1)$. 
\begin{enumerate}\item[(a)] Under conditions \eqref{eq:ordersA} and  $ d\le \eta \sqrt { \log n / \log \log n}$,  
w.h.p.\  
\be\la{propo:SBM1}
\|A \| \leq (1 + \eps) \sqrt L_1 + C \sqrt{ n \pma}\,.
\ee
 Under the additional conditions $d_i = d$ for all $i$ and $ n\pma \le \eta L_1$,  we have for all $k \in [n^{1-\eps}]$
\be\la{propo:SBM2}\lam_k(A)\sim \sqrt{ L_k} \AND \lam_{n+1-k} (A)\sim  - \sqrt{ L_k}\,.\ee
\item[(b)] Under conditions \eqref{eq:ordersA}, \eqref{541714h},  and  $ d\ge \eta^{-1} \sqrt { \log n / \log \log n}$,  
\be\la{propo:SBM3}
\forall i=1, \ld, k^+\,, \quad \lam_i(A)\sim  \lam^+_i\,, \qquad\qquad \forall i=1\,, \ld, k^-\,,\quad \lam_{n+1 - i}(A)\sim  \lam^-_{i}\ee and w.h.p., \be\la{propo:SBM4}\max\{ |\lam_{k^++1}(A)|, |\lam_{n-k^-}(A)|\} \le (1 + \eps) \sqrt L_1 + C \sqrt{ n \pma}\,.
\ee
 Under the additional   conditions $d_i = d$ for all $i$ and $ n\pma \le \eta L_1$, we have for all $k \in [n^{1-\eps}]$
\be\la{propo:SBM5}\lam_{k^++k}(A)\sim \sqrt{ L_k} \AND \lam_{n+1-(k^- + k)} (A)\sim  - \sqrt{ L_k}\,.\ee
\end{enumerate}
\en{propo}

We remark that in the case (a) of small degree, the nontrivial eigenvalues $\lambda^\pm_i$ of $\E A$ do not give rise to corresponding eigenvalues of $A$, and $A$ may be regarded as a perturbation of $\ul A$. In contrast, in the case (b) of large degree, the nontrivial eigenvalues $\lambda^\pm_i$ of $\E A$ gives rise to associated outlier eigenvalues of $A$, and $A$ may be regarded as a perturbation of $\E A$. Hence, the spectrum of $A$ retains some information about the spectrum of $\E A$ if and only if $ d\gg \sqrt { \log n / \log \log n}$.

 \beg{rmk}Remarks \ref{rem:pospart2} and \ref{rmk:Xi} also hold for the eigenvalue counting measure of $A$. See Figure \ref{Fig:histo_edge} for an illustration. \en{rmk}

\begin{proof}[Proof of Proposition \ref{propo:SBM}]
The proof is a simple consequence of the results of the preceding subsections and of Weyl's inequalities. More specifically, use \cite[Corollary III.2.6]{Bhatia} to prove \eqref{propo:SBM1}, \eqref{propo:SBM2}, \eqref{propo:SBM3} and \eqref{propo:SBM4} considering $A$ as a perturbation of $\uA$ in (a) and $A$ as a perturbation of $\E A$ in (b).   To prove the first part of \eqref{propo:SBM5} (the second part can be proved in the same way),  
note that by \cite[Corollary III.2.3 and Exercise III.2.4]{Bhatia}
we have, for any $k$,  
$$\lam_{k+2k^++k^-}(\uA)\le \lam_{k^++k}(A)\le \lam_k(\uA)$$
and use Corollary \ref{cor:barA}. 
\end{proof}

\subsection{Behaviour of the largest degrees}
In our regime of interest, the largest degrees of the graph play a key role in the analysis of the largest eigenvalues of the adjacency matrix.  We now describe their asymptotic behavior.

  \beg{proposition}\la{propo_position_degree_max}
  Let $\eps \in (0,1)$. Under \eqref{eq:ordersA}, w.h.p., for any $k \in [n^{1-\eps}]$
  $$ D_k^{\da} \le (1+\eps)L_k \,.$$ 
 \en{proposition}

For any   $x>0$, we introduce the sets \be\la{def:usoleq0V}\mc{V}_{ \geq  x} \deq \{i \in [n] \ste D_i \geq  x\}\,, \qquad
\mc{V}_{= x} \deq \{i \in [n] \ste D_i = x\}\,.
\ee These sets are related to the ordered degrees through
\be\la{641710h}D^{\da}_k \geq x\iff \# \mc{V}_{\ge x} \geq k\,.\ee
Let us introduce the function $f$ on $(d,\infty)$ defined by
\begin{equation}\label{eq:deff}
f( x) =f_d(x)\deq x \log \PAR{ \frac x   d } - ( x-d) - \log \sqrt{2 \pi x }\,. 
\end{equation}
If $Y$ is a Poisson random variable with mean $d$, for a large integer $x$ Stirling's approximation gives
\begin{equation}\label{eq:fpoi}
\p ( Y = x ) =\me^{- f (x)  + O ( x^{-1})}\,. 
\end{equation}
We shall in fact prove that, 
roughly speaking, under condition \eqref{eq:ordersA}, we have  \be\la{641710h10} \#\mc{V}_{\ge x} \lesssim n\me^{f(x)}\,,\ee  which, under the additional assumption $d_i=d$ for all $i$, can be strengthened to  \be\la{641710h11}  \#\mc{V}_{\ge x} \approx n\me^{f(x)}\,.\ee  This leads us to introduce,
  for $k \in [n]$, the solution $\Del_k$ of the equation
$$ f(\Delta_k) =\log (n/k) \AND  \Delta_k \geq  d\,.$$
This solution is unique and satisfies
\begin{equation}\label{eq:asympDL}
 \Delta_k \sim L_k\,.
\end{equation}
(See Lemma \ref{LCDBMLEV} below for the full details.) The combination of the characterization \eqref{641710h} of the largest degrees with estimates \eqref{641710h10}, \eqre{641710h11} and \eqref{eq:asympDL} naturally leads to the estimates given in Propositions \ref{propo_position_degree_max} and \ref{27516}.

In the special case of a homogenous Erd\H{o}s-R\'enyi graph, the next proposition is essentially contained in \cite[Chapter 3]{MR1864966}. Hence, our next result may be viewed as a generalization of this result to the inhomogeneous case. It is a more precise version of Proposition \ref{propo_position_degree_max} under the additional assumption that all vertices have the same mean degree.
 \beg{propo}\la{27516}  Let $\epsilon \in (0,1)$. \begin{itemize}
\item[(a)] For any $k \in [n^{1-\eps}]$ there exists a deterministic    $\Delta'_k  \in \{  \lfloor \Delta_k \rfloor ,  \lceil \Delta_k \rceil \}$ \st 
under the conditions $d_i = d$ for all $i$ and \eqref{eq:ordersA}, w.h.p., for any $k \in [n^{1-\eps}]$,  we have \begin{enumerate}\ite $D^\da_{k} \in \{ \Delta'_k  ,  \Delta'_k   - 1 \}$,
\ite $D^\da_{k} = \lfloor \Delta_k \rfloor$ when $\mathrm{dist}( \Delta_k , \mathbb N) \geq \eps$.\end{enumerate}
\item[(b)]
Under the conditions $d_i = d$ for all $i$ and \eqref{eq:ordersA}, for all integers $t \in [\eps \Delta_1, \Delta_1 - \eps]$, 
 \be\la{estimate:cardVka} 
 \# \mc{V}_{\geq t}  \sim   \# \mc{V}_{= t}    \sim      n \me^{-f(t)}\,.
\ee
Under the same conditions, if $t \geq \Delta_1 + \eps$ then, w.h.p., $\# \mc{V}_{\geq t} = 0$.
\end{itemize}
\en{propo}

 An immediate corollary of Proposition \ref{27516} and \eqref{eq:asympDL} is the following. 
 
 \begin{cor}\label{cor:equivd}
 Let $\eps \in (0,1)$. Under the conditions $d_i = d$ for all $i$  and \eqref{eq:ordersA}, for all $k \in [n^{1-\eps}]$,   $$ D_k^{\da} \sim L_k \,.$$ 
  \end{cor}

  \beg{rmk}[Lack of limit point process of largest degrees]\la{noPoisson} Proposition \ref{27516} (b) shows that, perhaps surprisingly, there is no Poisson point process at the right edge of the multiset of degrees of  $G$. There is instead a sharp transition at $\Delta_1$: for any integer $ t \leq \Delta_1-\eps$, \whp  the number of  vertices with degree $t$ is $\gg 1$ and for any integer $  t \ge \Delta_1+\eps$, \whp there is no vertex with degree $t$.
 \en{rmk}

\section{Estimation of the largest degrees and comparison with the eigenvalues} \label{sec:proof}

The rest of this paper is devoted to the proofs of our main results.

Throughout this section we use the following conventions about convergence of deterministic quantities. Let $u$ and $v$ be deterministic quantities depending on $n$ and $(p_{ij})_{i,j\in [n]}$. We write $u = o(v)$, or, equivalently, $u \ll v$, whenever $u/v \to 0$ as $n \to \infty$ and $\eta \to 0$, uniformly in $(p_{ij})_{i,j\in [n]}$ satisfying \eqref{eq:ordersA} and all parameters except $\epsilon$. We remark that such a convergence can always be upgraded to a quantitative convergence using some admissible error function from Definition \ref{def:error_control}, but for the sake of simplicity we shall not do this.

\subsection{Largest degrees: proof of Proposition \ref{propo_position_degree_max} and   Proposition  \ref{27516}}

Recall that the function $f$ was defined in \eqref{eq:deff}.

\beg{lem}\la{LCDBMLEV} Let  $\epsilon \in (0,1)$.
 For $n$ large enough and $\eta$ small enough, under condition \eqref{eq:ordersA}, for any $k \in [n^{1-\eps}]$, there exists a unique solution $\Delta_k$ of the equation
 \be \la{def:usoleq0} f(\Delta_k) =\log (n/k) \AND  \Delta_k \geq  d\,.\ee
 Moreover, under condition \eqref{eq:ordersA}, for any $k \in [n^{1-\eps}]$, \begin{equation}\label{def:usoleq1} 
 \Delta_k \sim L_k \,.
 \end{equation}
 \en{lem}

\bpr   The function $f$ is increasing  on $(d,\infty)$ (indeed, $f'(u)=\log(u/d)+1/(2u)$) and satisfies  \be\la{f:est1}  f(u) =  u\log(u/d) + O ( u)\qquad \trm{ as $u\to\infty$}\,,\ee so that for $n$ large enough, $\Del_k$ is well defined for any $1 \leq k \leq n^{1 - \eps}$. Moreover, (unconditionally on $d$), we have   \be\la{f:est2}  f(u+x)=f(u)+x\log (u/d)+O ( x  / u + x^2 / u)\,. \ee  
Indeed,  \be f(u+x)-f(u)=(u+x)\log\lf(1+\f{x}{u}\ri)+x\log \PAR{ \frac u d}-x+\ff{2}\log\pB{1+\f{x}{u}}\,.
\ee

Let us now prove \eqref{def:usoleq1}. As both $\Del_k=\Del_k(n,d)$ and $L_k=L_k(n,d)$ are deterministic and depend only on $n$ and $d$, by Definition \ref{def:error_control}, \eqref{def:usoleq1} reads \be\la{04041712h}\lim_{\substack{n\to\infty\\�\eta\to 0}}\sup_{\ka\le d\le \eta\log n}\sup_{k\in [n^{1-\eps}]}\lf|\f{\Del_k(n,d)}{L_k(n,d)}-1\ri|=0\,.\ee 
If it were not the case, there would be an infinite set $I$ of positive integers and some sequences $(\eta_n)_{n\in I}, (d_n)_{n\in I}, (k_n)_{n\in I}$ satisfying $\ka\le d_n\le \eta_n \log n$, $1\le k_n \le n^{1-\eps}$ and   $\eta_n\to 0$ as $n\in I$ tends to infinity, 
\st \be\la{441714h22}\lf|\f{\Del_{k_n}(n,d_n)}{L_k(n,d_n)}-1\ri|>c\ee  for some positive constant $c$. 
Let us drop the index $n$ from the notation. 
 One first verifies that $\Delta_k/d \to \infty$ as $n \in I$ grows 
 (by a simple argument by contradiction using \eqref{def:usoleq0}). Then, introduce $\vfi>0$ \st 
$\Delta_k=\vfi L_k.$ By \eqre{f:est1}, we have
\begin{align*}
 \log (n/k)&=f \lf(\Delta_k\ri)\\ &\sim 
\Delta_k\log(\Delta_k/d)\\ &=\vfi \f{\log (n/k)}{\log((\log n)/d)}\pb{\log(\log (n/k)/d)+\log\vfi -\log \log((\log n)/d)}\,.
\end{align*}
By assumption,  $\eps \log n \leq  \log (n/k)  \leq \log n$  so that
 \be\la{1291112}\ff{\vfi}\sim 1+\f{\log \vfi}{\log((\log n)/d)}+o(1)\,.\ee
On easily deduces from \eqre{1291112} that 
 $\vfi$ is bounded away from $0$ and $\infty$,   and then that $\vfi$ tends to one, which contradicts \eqre{441714h22}. Thus \eqref{04041712h}, hence also \eqre{def:usoleq1}, are  true.  \epr

 \beg{lem}\la{LCDBMLEVBIS}
Suppose that $n$ is large enough and $\eta$ small enough, and that $d$ satisfies condition \eqref{eq:ordersA}, so that $\Delta_1$ is well defined (see Lemma \ref{LCDBMLEV}). Let $q_1, \dots, q_n > 0$ satisfy $d = \sum_i q_i$, and let $X$ be a random variable with law $\BIN(q_1, \ldots , q_n)$.
Suppose that
$q_{\max} \deq \max_i q_i \leq d / (\log n)^{5/2}$. Then  for any  $u \leq \Delta_1$ and $x$ such that $x^2\leq u$ and $u+x \geq 2d $ is integer,  
\be\la{AsexpTail}
\p(X\ge u+x)=\PAR{ 1+ \delta} \PAR{\frac{u}{d}}^{-x}\me^{  -f(u)}\,.
\ee 
where $\delta = O \PAR{  (d   + x^2 )/u  +    q_{\max} u^{ 5 /2} / d  }$.
\end{lem}

\bpr
 By Theorem \ref{th:tail} below, $\p(X\ge u +x)=\PAR{ 1+ \delta} \p( Y =   u +x)$ where  $Y$ has Poisson distribution with mean $d$ and $\delta = O \PAR{  d /u   +    q_{\max} u ^{ 5 /2} / d  }$. Now, 
by \eqref{eq:fpoi}, we have 
 \be
  \p(Y =   u +x)  =   \me^{-f( u +x ) + O (1 / u)}\,.
 \ee
 Then, the estimate \eqre{f:est2}  allows to conclude.\epr

 We are now ready to prove Proposition \ref{propo_position_degree_max}.

 \begin{proof}[Proof of Proposition \ref{propo_position_degree_max}] Let $k \in [n^{1-\eps}]$. It is sufficient to prove that \whp we  have  $D_k ^\da \leq ( 1+ \eps) L_k$. Indeed, the inversion of \whp and {\it for all $k \in [n^{1-\eps}]$} is straightforward using $k = n^{1 - \ell \eps}$ for each $\ell \in [1/\epsilon]$, as all error terms $o(1)$ in what follows are terms tending to zero uniformly in $k \in [n^{1-\eps}]$ as $n\to\infty$ and $\eta\to 0$. 
 
 We note that 
$D^{\da}_k \geq t$ is equivalent to $\# \mc{V}_{\ge t} \geq k$ (the set $\mc{V}_{\ge t} $ has been defined in \eqref{def:usoleq0V}).  By \eqref{f:est2}, we have 
\begin{equation}\label{eq:deltakx}
  n \me^{-f( \Delta_k + x)} =  k  (  1+ o(1)) ( \Delta_k / d )^{-x}\,.
\end{equation}
 Thus \eqref{eq:deltakx} applied to $x = \lfloor \Delta_k \rfloor  +2 - \Delta_k $ implies that  $n\me^{-f(  \lfloor \Delta_k \rfloor  +2 )} = o(k)$. By \eqref{AsexpTail}  in  Lemma \ref{LCDBMLEVBIS}, it follows that $\E \#\mc{V}_{\ge t}=o(k) $ if $t =  \lfloor \Delta_k \rfloor  +2$. It remains to use \eqref{def:usoleq1} and Markov's inequality. 
\epr
 
Our proof of Proposition  \ref{27516} will require a sharp bound on the variance of $\#\mc{V}_{\ge t} $. 

\beg{lem}\la{estimate_variance}   Let  $D_i$ denote the degree of the vertex $i$ in the graph $G$. Then any integer $t\ge 0$, $$
\Var (\#\mc{V}_{\ge t} )\le\E \#\mc{V}_{\ge t}  +3d n \max_i\p(D_i\ge t-1)^2 \,.$$
\en{lem}

\bpr For ease of notation, we set $q_i := \p(D_i\ge t)$. Since $\#\mc{V}_{\ge t} = \sum_{i=1}^n \one_{D_i \geq t}$, we have $\E \#\mc{V}_{\ge t}   = \sum_{i=1}^n q_i$ and   \begin{align*}
  \Var(\#\mc{V}_{\ge t} ) 
  &=\sum_{i,j}\Cov(\one_{D_i\ge t}, \one_{D_j\ge t})\\
  &=\sum_{i}q_i(1-q_i)+\sum_{i\ne j}\Cov(\one_{D_i\ge t}, \one_{D_j\ge t})\,.  \end{align*}
Hence  it suffices to prove   that for $i\ne j$,  \bes\la{66161}\Cov(\one_{D_i\ge t}, \one_{D_j\ge t})\le  3 p_{ij} \max_i\p(D_i\ge t-1)^2 \,.\ees
 Let us fix $i\ne j$. We have $D_i=\sum_k A_{ik}$ and $D_j=\sum_{k}A_{jk}$. We  introduce the events 
 \begin{align*} & E \deq \lf\{\sum_{k\ne j} A_{ik} \ge t, A_{ij} = 0  \ri\}\,,\qquad E' \deq \lf\{\sum_{k} A_{ik} \ge t \ri\} \,,\\
& F \deq \lf\{\sum_{k\ne i} A_{jk} \ge t , A_{ij} = 0 \ri\}\,, \qquad
F' \deq \lf\{\sum_{k} A_{jk} \ge t\ri\}\,.
\end{align*}
Then $E\subset E'$, $F\subset F'$ and $(E'\backslash E)  \cap F =(F'\backslash F)  \cap E = \emptyset$, the latter follows from $E' \bck E = \{  \sum_{k\ne j} A_{ik} \ge t-1 , A_{ij}  =1 \} $ (and similarly for $F' \bck F$).  Thus by Lemma \ref{lem:covBernoulli} and the independence of the events $\{\sum_{k\ne j} A_{ik} \geq s \}$ and $\{\sum_{k\ne i} A_{jk} \geq s'\}$
\begin{align*}
\Cov(\one_{D_i\ge t}, \one_{D_j\ge t})&=\Cov(\one_{E'}, \one_{F'})\\ &\le  \p\pb{(E' \bck E) \cap (F'\bck F)} +  \p(F'\bck F)  \p(E) +  \p(E'\bck E) \p (F)\\ &\le 3  p_{ij} \p\pbb{\sum_{k\ne j} A_{ik} \ge t-1} \p\pbb{\sum_{k\ne i} A_{jk} \ge t-1}\,,
\end{align*}
which allows to conclude. \epr

 We are ready to prove Proposition  \ref{27516}. 
 
 \begin{proof}[Proof of Proposition \ref{27516}]
First we remark that the inversion of \whp and {\it for all  for all integers $t \in [\eps \Delta_1, \Delta_1 - \eps]$} for (b) or \emph{for all $k \in [n^{1 - \epsilon}]$} for (a) can be treated as in the proof of Proposition \ref{propo_position_degree_max}.

\begin{itemize}
\item[(b)]
By \eqref{AsexpTail}  in  Lemma \ref{LCDBMLEVBIS}, if $t \geq \Delta_1 + \eps$, $ \p ( \#\mc{V}_{\ge t} \geq 1) \le \E \#\mc{V}_{\ge t} \leq (1 + \eps) ( d / \Delta_1)^{\eps}$ tends to $0$.  Moreover, in the regime $\eps \Delta_1 \leq t \leq \Delta_1 - \eps$, as $n\me^{-f(t)}$ goes to infinity, to prove the left-hand side of \eqre{estimate:cardVka}, by Markov's inequality it suffices to prove  that $\E \#\mc{V}_{\ge t} = n\me^{-f(t)}(1+o(1))$ and $\Var ( \#\mc{V}_{\ge t} ) = n\me^{-f(t)}(1+o(1))$, which follows directly from \eqref{AsexpTail}  in  Lemma \ref{LCDBMLEV} and  Lemma \ref{estimate_variance}. 
 
It remains to prove that $  \# \mc{V}_{= t}  \sim n\me^{-f(t)}$ when $\eps \Delta_1 \leq t \leq \Delta_1 - \eps$. We note that $$\#\mc{V}_{= t}= \#\mc{V}_{\ge t}-\#\mc{V}_{\ge t+1}\,.$$
From what precedes, it suffices to check that   $\# \mc{V}_{t+1}=o(n\me^{-f(t)}) $. The latter is a consequence of  \eqref{AsexpTail}  in  Lemma \ref{LCDBMLEV} which implies that $\E \# \mc{V}_{t+1} \leq (1 + o(1)) ( d /  t )  n\me^{-f(t)}$. 

\item[(a)]
We note that for any $k,t$, the claim 
$D^{\da}_k \geq t$ is equivalent to $\# \mc{V}_{\ge t} \geq k$. Thus \eqref{eq:deltakx} applied to $x = \lfloor \Delta_k \rfloor  -1  - \Delta_k $ and (b) imply that with high probability $D^{\da}_k \geq \lfloor \Delta_k \rfloor  -1$. Similarly,  \eqref{eq:deltakx} applied to $x = \lfloor \Delta_k \rfloor  +2 - \Delta_k $ and (b) imply with high probability  $D^{\da}_k \leq \lfloor \Delta_k \rfloor  +1$. Moreover, if $\lfloor \Delta_k \rfloor  + 1 -  \Delta_k \geq   \eps$, then with high probability $D^{\da}_k \leq \lfloor \Delta_k \rfloor $, while if $\lfloor \Delta_k \rfloor  -  \Delta_k \leq   -  \eps$ then  with high probability $D^{\da}_k \geq \lfloor \Delta_k \rfloor $.  Note that either $\lfloor \Delta_k \rfloor  -  \Delta_k \leq  - \eps$ or $\lfloor \Delta_k \rfloor   +1   -  \Delta_k \geq  \eps$ holds when $\eps = 1/2$.  \qedhere
\end{itemize}
\epr

\subsection{Proof of Theorem \ref{thalphall12}}

First it is easy to see,  by Weyl's inequality, that we may assume without loss of generality that $p_{ii} = 0$ for all $i \in [n]$.  As pointed in introduction, our strategy is to describe the graph spanned by the vertices of high degree. We start with a deviation inequality on the degrees. Define
\begin{equation} \label{def_h}
h(x) \;\deq\; (1+x) \log (1+ x) - x\,.
\end{equation}

\begin{lem}\label{le:deg}
For distinct $i_1, \ldots , i_k \in [n]$ and $t \geq 0$ we have
\begin{equation}\label{eq:devdeg}
\p \left( \sum_{\ell=1}^k D_{i_\ell} \geq  k ( d +  t)  \right) \leq \exp \left( - k d h \left( \frac t   d \right) + k^2 \pma \left(\frac  t  d +
 1 \right)^2 \right)\,.
\end{equation}
In particular, if $d \ll t$ and $k t  / \log (t / d) \ll d^2 / \pma$, we have 
\begin{equation}\label{eq:devdeg0}
\p \left( \sum_{\ell=1}^k  D_{i_\ell} \geq    k t \right) \leq \exp \left( - ( 1+ o (1))  k t \log \left( \frac t   d \right) \right)\,. 
\end{equation}
\end{lem}

\bpr
We have
$$
\p \left( \sum_{\ell=1}^k D_{i_\ell} \geq  k ( d + t)  \right) \leq \p \left( \sum_{\ell=1}^k D_{i_\ell}  - \E D_{i_\ell} \geq  k t\right)\,. 
$$
Now, 
$$
\sum_{\ell=1}^k D_{i_\ell} =2  \sum_{\ell=1}^k \sum_{j = \ell +1}^{k}A_{i_\ell i_{j}} +  \sum_{\ell=1}^k \sum_{ j \notin \{i_1 , \ldots , i_k \} }A_{i_\ell j} \eqd 2 S_1 + S_2\,,
$$
where $S_1$ and $S_2$ are independent. From Chernov's bound, for any $\lambda \geq 0$, 
\begin{equation*}
 \p \left( \sum_{\ell=1}^k D_{i_\ell} - \E D_{i_\ell}  \geq    k t \right) \leq \exp ( - \lambda k t + \alpha \phi( \lambda) +  \beta \phi ( 2 \lambda) ),
\end{equation*}
where $\phi( \lambda) = \me^{\lambda} - \lambda - 1$, $$
\alpha =  \sum_{\ell=1}^k \sum_{ j \notin \{i_1 , \ldots , i_k \} }p_{i_\ell j} 
\AND 
\beta  = 2 \sum_{\ell=1}^k  \sum_{j = \ell +1}^{k} p_{i_\ell i_j} = \sum_{\ell=1}^k  \sum_{j = 1}^{k} p_{i_\ell i_j}\,.
$$
By hypothesis, $\alpha + \beta \leq k d$ and $\beta \leq k^2  \pma$. We take $\lambda = \log ( t / d + 1)$ and use that $\phi( 2\lambda) \leq e^{2\lambda}$, we arrive at \eqref{eq:devdeg}. The second claim \eqref{eq:devdeg0} is an immediate consequence of the fact that the function $h$ from \eqref{def_h} satisfies  $h (x )  \sim x \log x$ when $x$ goes to infinity. 
\epr

\begin{lem}\label{le:nena}
There exists a constant $c >0$ such that the following holds.
Let $\eps \in (0,1)$ and $t = \eps L_1$.  For $S \subset \{1 , \ldots , n\}$,  let $\cN(S) = \{ j : A_{ij} = 1 \hbox{ for some } i \in S \}$ denote the set of neighbours of elements in $S $.  Then, under \eqref{eq:ordersA}, \whp for any $i \in \mc{V}_{\geq t}$ we have
\begin{equation*}
\# \qb{\cN(i) \cap \left( \mc{V}_{\geq t}   \cup \cN(\mc{V}_{\geq t} \backslash \{ i \} )\right)}  \leq c / \eps\,.
\end{equation*}
\end{lem}

\bpr
Fix an integer $k$. 
 Let $P(1)$ be the probability that there exists a  vertex of $\mc{V}_{\geq t}$ which is neighbour to at least   $k$ other elements of  $\mc{V}_{\geq t}$. We have 
\begin{align*}
P(1) &\leq  \sum_{\substack{ i_0, \ldots , i_k \\ \trm{distinct}}} \p \PAR{ D_{i_0} \geq t \; \hbox{ and } \; 
 \forall \ell \in \{ 0, \ldots , k \}   : D_{i_\ell} \geq t , A_{i_0 i_\ell}   = 1} \\
& \leq  \ \sum_{\substack{ i_0, \ldots , i_k \\ \trm{distinct}}}  \p \PAR{D_{i_0} \geq t - k }  \p \PAR{ \forall \ell \in \{ 1, \ldots , k \}   : D_{i_\ell} \geq t -1} \prod_{\ell = 1}^{k } p_{i_0 i_\ell} \end{align*}
Since for any fixed $i_0$ we have \be\la{23101612h}\sum_{\substack{ i_1, \ldots , i_k \\ \trm{distinct}}}\prod_{\ell = 1}^{k } p_{i_0 i_\ell}\le d^k \ee we deduce that \begin{align*}
P(1) 
& \leq n d^k \max_i  \p \PAR{ D_{i} \geq t - k }  \max_{\substack{ i_1, \ldots , i_k \\ \trm{distinct}}}\p \PAR{ \forall \ell \in \{ 1, \ldots , k  \}   : D_{i_\ell} \geq t -1} \\
& \leq n d^k \exp \left( - (1+ o(1)) (k+1) t \log \frac{ t}{d}\right)
\end{align*}
if $t \gg d $ and $t\pma\ll d^2\log(t/d)$ (from \eqref{eq:devdeg0}). For $t = \eps L_1 \gg d$ and $k$   fixed   such that $k +1 >  a / \eps$ with $a >1$. We find 
$$
P(1) \leq  n^{ 1  - a (1 + o(1)) }   = o(1)\,. 
$$ 
Similarly, let $P(2)$ be the probability that  there exists a vertex $i \in \mc{V}_{\geq t}$ which is neighbour  to at least $k$ elements of  $\cN ( \mc{V}_{\geq t} \backslash\{ i \} )$. Then
\begin{align*}
P(2) \leq & \sum_{s = 1}^{k} \sum_\tau \sum_{\substack{ i_0, \ldots , i_s , j_1, \ldots , j_k \\ \trm{distinct}}} \p \PAR{ D_{i_0}, \ld, D_{i_s} \geq t \hbox{ and } \;  \forall \ell= 1, \ldots , k  :   A_{i_0 j_{\ell}} = A_{ j_{\ell}i_{\tau(\ell)}}   = 1 }\,,
\end{align*}
where the second sum is over all surjective maps $\tau : \{ 1, \ldots , k \} \to \{ 1, \ldots , s \}$. We deduce that 
\begin{align*}
P(2) & \leq   \sum_{s = 1}^{k} \sum_\tau \sum_{\substack{ i_0, \ldots , i_s , j_1, \ldots , j_k \\ \trm{distinct}}} \p \PAR{ D_{i_0} \geq t  - s } \p \PAR{ \forall r \in \{ 1, \ldots , s  \}   :D_{i_r} \geq t -\#\tau^{-1} (\{r\})} \prod_{\ell= 1 }^k p_{i_0 j_\ell}  p_{ j_\ell i_{\tau(\ell)}} \,. \end{align*}
Now, note that for any fixed surjective map $\tau : [k] \to [s]$, \bgt\ite  for any $i_1, \ld, i_s$, we have $$\lf\{ \forall r \in \{ 1, \ldots , s  \}   :D_{i_r} \geq t -\#\tau^{-1} (\{r\})\ri\}\subset  \lf\{\sum_{r=1}^s D_{i_r} \ge   st -k\ri\}\,;  $$\ite 
for any fixed $j_1, \ld, j_k$, we have,  as in \eqre{23101612h},   $$ \sum_{\substack{  i_1, \ldots , i_s \\ \trm{distinct}}} \prod_{\ell= 1 }^k   p_{ j_\ell i_{\tau(\ell)}} \le d^s\pma^{k-s}\,.$$
\ent We deduce, using \eqre{23101612h} and \eqref{eq:devdeg0} again, that 
\begin{align*}
P(2) 
& \leq  \sum_{s = 1}^{k}  \sum_\tau n d^{k+s}  \pma^{k-s}\max_i  \p \PAR{ D_{i} \geq t - k }  \max_{\substack{ i_1, \ldots , i_s \\ \trm{distinct}}}\p \PAR{  \sum_{ \ell = 1}^s  D_{i_\ell} \geq  s t -k } \\
& \leq  n d^k  \sum_{s = 1}^{k}  {k \choose s} k^{k-s} d^s \pma^{k-s} \me^{ - (1+ o(1)) (s+1) t \log \left( \frac{ t}{d} \right) }\,.
\end{align*}
 The $o(1)$ is uniform over $1 \leq s \leq k$. Hence, 
$$
P(2) \leq n d^{k}   \me ^{  - (1+ o(1)) t \log \left( \frac{ t}{d} \right) } \left(d \me ^{  - (1+ o(1))  t \log \left( \frac{ t}{d} \right) }  + k \pma    \right)^k \,. 
$$
As above for $t = \eps L_1 \gg d$ and $k$ a fixed integer such that $k +1 >  a / \eps$ with $a >1$. We find 
$
P(2) = o(1). 
$ This concludes the proof of the first claim of the lemma. 
\epr

\begin{Def} \label{def:star_graph}
A \emph{star graph with central degree $D$} is the graph with vertex set $\{0,1,\dots, D\}$ and edge set $\{\{0,1\}, \{0,2\}, \dots, \{0,D\}\}$.
\end{Def}

We may now prove Theorem \ref{thalphall12}.

  \begin{proof}[Proof of Theorem \ref{thalphall12}] Let $0 < \delta <  \eps  $ and set $t: = \delta L_1$. By Proposition \ref{propo_position_degree_max}, \whp for any $k \in [n^{1-\eps}]$, 
  $$
  D^\da_k \geq t\, .
  $$
Let $G_\star$ be the graph  obtained from $G$ as follows. The vertex set of $G_\star$ is $[n]$. The edge set of $G_\star$ is the set of edges $\{ i,j \}$ of $G$ (i.e.\ $A_{ij} = 1$)  such that $i \in \mc{V}_{\geq t}$ and $j \notin  \mc{V}_{\geq t}  \cup \cN(\mc{V}_{\geq t} \backslash \{ i \} ) $  (where the notation $\cN(\cdot)$ was introduced in Lemma \ref{le:nena}). By construction, $G_\star$ is a disjoint union of isolated vertices and of star graphs with central degrees $D_i^\star$, $i \in \mc{V}_{\geq t}$. By Lemma \ref{le:nena}, \whpv the central degrees of the stars satisfy $D_i - c / \delta \leq D_i^\star \leq  D_i$. Let $A_\star$ be the adjacency matrix of $G_\star$ and  let $A' = A - A_\star$ be the adjacency matrix of $G \backslash G_\star$.

By Lemma \ref{StarGraph}, the nonzero eigenvalues of $A_\star$ are $\pm\sqrt{D^\star_i}$, $i \in \mc{V}_{\geq t}$.  From what precedes, \whp for all $i \in \mc{V}_{\geq t}$, 
\be\la{eq231019h30}
\left| \sqrt{D^\star_i} - \sqrt{D_i}  \right| = \f{D_i-D_i^\star}{\sqrt{D_i} + \sqrt{D_i^\star}}\leq \frac{c }{\delta \sqrt t}\,. 
\ee
 for $c$ the universal constant of Lemma \ref{le:nena}. 
 
 Besides,   
 \whp the maximal degree in $G \backslash G_\star$ is bounded by $\max( t  ,  c / \delta)$. Indeed, let $i\in [n]$ be a vertex. If the degree of $i$ in $G$ is $< t$, then there is nothing to prove (as the degree of $i$ in  $G\bck G_\star$ is bounded by its degree in $G$), whereas if $D_i\ge  t$, then the degree of $i$ in  $G\bck G_\star$ is $D_i-D_i^\star\le c/\delta$. 
 
By Proposition \ref{27516}, we know that \whpv the cardinal number of $ \mc{V}_{\geq t}$ is at most  $n^{1-\delta + o(1)}\le 10/\pma$, hence 
  by Theorem \ref{LeVer} of Le, Levina, and Vershynin, \whp
\be\la{eq8juin19h30}\|A'-\E A\|\le C((n\pma)^{1/2}+ (\delta L_1)^{1/2})\,.\ee

Then, one concludes thanks to Weyl's perturbation inequality (see e.g.\ \cite[Corollary III.2.6]{Bhatia}), noticing that the constants from \eqre{eq231019h30} and \eqre{eq8juin19h30} do not depend on the choice of $\delta>0$.
\epr

\section{Poisson tail aproximation} \label{sec:Poisson}

The following sharp tail asymptotic of  $\BIN( p_1, \ldots , p_n)$ is of independent interest. It is stronger than what can be deduced from Le Cam's inequality (see \cite{BHJBook}).
 \begin{Th}\label{th:tail}
Let $X$ with distribution $\BIN( p_1, \ldots , p_n)$, $d =p_1+\cd+p_n>0$ and $\pma = \max_i p_i \geq d / n$. Let $Y$ be a Poisson variable with mean $d$. There exists a universal constant $C >0$ such that for any integer $k$ satisfying  $ 2 d \leq  k \leq (  d  / \pma)^{2/5} / C$, we have
\begin{equation}\label{eq:gn20}
\ABS{ \frac{\p(X=k)}{\p ( Y = k) }  - 1 }  \leq C \frac{ \pma k^{\f 5 2}}{d} \,,
\end{equation}
and
\begin{equation}\label{eq:gn10}
\p(X > k) \leq  C \PAR{ \frac { d}{k}  +  \frac{ \pma k^{\frac 5 2}}{d}  }  \p ( X = k)\,.
\end{equation}
 \end{Th}

We first check that Theorem \ref{th:tail} holds for standard binomial variable.

 \beg{lem}[Tails of binomial laws]\la{Lem:BinomialTail}   Let $Z$ be distributed according to the binomial distribution with parameters $(n,d/n)$. There exists a universal constant $C >0$ such that   for any  integer $k$ with  $2d \leq k \leq \sqrt n / C $, 
\begin{equation}\label{eq:gn2}
\ABS{ \frac{\p(Z=k)}{\p ( Y = k) }  - 1 }  \leq \frac{ C k^2}{ n } \,,
\end{equation}
and
 \begin{equation}\label{eq:gn1}
\p(Z > k) \leq \frac {C d}{k}   \p ( Z = k)\,.
\end{equation}
\en{lem}
 
 \bpr We have $$
 \frac{\p(Z=k)}{\p ( Y = k) } 
 =\f{n!}{(n-k)!(n-d)^k}\lf(1-\f{d}{n}\ri)^n\me^d$$
 using $(n-k)^k\le \f{n!}{(n-k)!}\le n^k $, we get that 
 \be\la{279161} \PAR{ 1 - \f k n}^k\lf(1-\f{d}{n}\ri)^n\me^d\le\frac{\p(Z=k)}{\p ( Y = k) } 
 \le  \lf(\f{n}{n-d}\ri)^k \le \me^{kd/(n-d)}\,.\ee
Using $\log(1+x)\le x$, we get $\lf(\f{n}{n-d}\ri)^k\le \me^{kd/(n-d)}$. Then, it is easy to see that there is $C$ \st as soon as $d,k\le \sqrt{n}$ and $n\ge 2$, we have   $\me^{kd/(n-d)}\le 1+C\f{kd}{n}$, so that 
 \be\la{279164}\frac{\p(Z=k)}{\p ( Y = k) } 
  \le 1+C\f{kd}{n}\,.\ee
 For the lower bound, first note that there is $C$ \st  for any  $1\le k \leq n$, \be\la{279162}
1 - \f{ C k ^2} { n } \leq \PAR{ 1 - \f k n}^k \ee (indeed, it  comes down to prove that there is a constant $C$ \st for any $n\ge 1$, for $x\in (0, (Cn)^{-1/2})$, $\log(1-Cnx^2)\le nx\log(1-x)$, which is easily obtained thanks to the series expansion). Then note that there are $C,C'$ \st as soon as $d\le\sqrt{n}$ and $n\ge 2$,  \be\la{279163}\log\lf(1-\f{d}{n}\ri)+\f{d}{n}\ge -C\lf(\f{d}{n}\ri)^2\AND \me^{-nC(d/n)^2}  \ge 1-C'\f{d^2}{n}\,.\ee  From \eqre{279161}, \eqre{279162} and \eqre{279163}, we deduce that  \be\la{279165}\frac{\p(Z=k)}{\p ( Y = k) } 
  \ge \lf( 1-C\f{k^2}{n}\ri)\lf( 1-C'\f{d^2}{n}\ri).\ee
The claim \eqref{eq:gn2} 
then follows from \eqre{279164} and \eqre{279165}.

  Note that for any integer $j  \ge d$, $$\p(Z=j+1)=\p(Z=j)\f{(n-j)_+}{n-d} \, \f{d}{j+1}\le \f{d}{j+1}\p(Z=j)\,.$$
 We deduce that for $\eps \deq \f{d}{k}$,  $$\p(Z >  k)\le  \p(Z=k)\sum_{\ell\ge 1}\eps^\ell=\f{\eps}{1-\eps} \p(Z=k)\,.$$
 This concludes the proof.
 \epr

The classical Bennett's inequality gives a first tail bound for the distribution $\BIN( p_1, \ldots , p_n)$.

\beg{lem}[Half of Bennett's inequality]\label{Bennett} Let $X$ with distribution $\BIN( p_1, \ldots , p_n)$ and set  $d =p_1+\cd+p_n$.  Then for any $t\ge0$, we have   \begin{equation}\label{roy012}\p\PAR{X \ge d(1+t) }\le \exp\{-d h(t)\} 
\end{equation}
where $h$ was defined in \eqref{def_h}.
\en{lem}

 The next crucial lemma will use the Lindeberg's replacement principle to compare the distributions $\BIN( p_1, \ldots , p_n)$ and $\BIN(p, \ldots , p)$ with $p = \frac 1 n \sum_i p_i$.
 \begin{lem}[Comparison principle]\label{le:lind}
Let $P = \BIN(p_1, \ldots , p_n)$ and $Q = \BIN( d / n , \ldots , d / n )$ with $d  \deq  \sum_{i=1}^n p_i$. For any integer $ k$, we have 
$$
 |P(k) - Q(k) | \leq  2  d \pma   \me^{ - d h(( \frac{ k- 2 - d}{ d} )_+ )} \AND   |P(k,\infty) - Q(k,\infty) | \leq  d \pma   \me^{ - d h (( \frac{ k- 2 - d}{ d})_+ ) }  \,. 
$$
\end{lem}
 
 \bpr
 Fix the integers $1 \leq i,j \leq n$. We write $f(t) = P(t)$ and $f^{ij} (t)$ is the probability of $t$ for the distribution $\BIN ( (p_l)_{l \ne \{i,j\}})$.  We have 
 \be\la{299167}
 f(t) = (1-p_i) (1 - p_j) f^{ij} (t) + p_i ( 1- p_i) f^{ij} (t-1) + p_j ( 1- p_j) f^{ij} (t-1) + p_i p_j f^{ij} (t-2)\,. 
 \ee
 As $(1-p_i) (1 - p_j)+p_i ( 1- p_i)+p_j ( 1- p_j)+p_i p_j=1$, we first deduce from \eqre{299167} that  for any $k\ge 0$, \be\la{289161} P(k,\infty) = \sum_{t\ge k} f(t)\ge\sum_{t\ge k}f^{ij}(t)\,. \ee
 We also deduce  from \eqre{299167} that  $f(t)$ is a polynomial of order $2$ in $(p_i,p_j)$. Interestingly, the term of order $1$ is symmetric and equal to 
 $$
 (p_i + p_j) (f^{ij} (t-1) - f^{ij} (t))\,. 
 $$
 The second order term is equal to 
 $$
 p_i p_j  (f^{ij} (t) - 2f^{ij} (t-1) + f^{ij} (t-2) )\,. 
 $$
  Now,  let $(q_1, \ldots , q_n)$ be such that $p_l = q_l$ for $k \ne \{ i, j \}$, $q_i+ q_j = p_i + p_j$. Hence, if $g(t)$ is the probability of $t$ for $\BIN(q_1, \ldots , q_n)$, we find
  $$
  f(t) - g(t) =  (p_i p_j  - q_i q_j)  (f^{ij} (t) - 2f^{ij} (t-1) + f^{ij} (t-2) )\,. 
  $$
  and 
  $$
  \sum_{ t\geq k} f(t) - g(t) = (p_i p_j  - q_i q_j)  (f^{ij} (k-2) - f^{ij} (k-1) ) \,.
  $$
  Assume that $p_i p_j $ and $q_i q_j$ are both bounded by $d\pma / n$. Then from the previous equations and from  \eqre{289161}, we deduce that
 $$
 | f(k) - g(k)|  \leq  2 \frac{d\pma}{n}  P( k - 2, \infty) \AND  \absbb{ \sum_{t \geq k} f(t) - g(t)}  \leq   \frac{d\pma}{n}  P( k - 2, \infty)\,.
  $$
   By Lemma \ref{Bennett}, we find that 
\begin{equation}\label{eq:lind}
   | f(k) - g(k)| \leq  2 \frac{d\pma}{n}    \me^{ - d h (( \frac{ k - 2 - d}{ d} )_+ )}  \AND  \absbb{ \sum_{t \geq k} f(t) - g(t)}  \leq   \frac{d\pma}{n}  \me^{ - d h(( \frac{ k - 2 - d}{ d} )_+  )}\,. 
\end{equation}
  Now, if $(p_1, \ldots , p_n) \ne ( d / n , \ldots , d / n)$ then there exists $(i,j)$ such that $p_i < d / n   < p_j$ (since the average $d/n$ is in the convex hull of $(p_1, \ldots , p_n)$).  We consider $(q_1, \ldots, q_n)$ as above such that $q_i = d / n $ and $  p_i < q_j  = p_i + p_j  - d / n < p_j   $. Then the bound \eqref{eq:lind} applies here and $\qma = \max_k q_k \leq \pma$. We may thus repeat the same operation to $p^1 = (q_1, \ldots , q_n)$ and  get $p^2$ and so on. After  $ m \leq n$ iterations, we arrive at $p^m = (d/n, \ldots , d/n)$. Summing the $m$ times \eqref{eq:lind} gives 
  $$
  |P(k) - Q(k) | \leq  2 m  \frac{d\pma}{n}    \me^{ - d h(( \frac{ k - 2 - d}{ d} )_+  )}  \leq 2  d \pma   \me^{ - d h (( \frac{ k - 2 - d}{ d} )_+  )}\,.
  $$
This gives the first claim. The same argument applied to the right-hand side of  \eqref{eq:lind} gives the second claim.  \epr
 
We are finally ready for the proof of Theorem \ref{th:tail}

\begin{proof}[Proof of Theorem \ref{th:tail}] 
Let $Z$ be a random variable with distribution $\BIN( d / n , \ldots , d/n)$. In view of Lemma \ref{Lem:BinomialTail}, it sufficient to prove (up to adjusting the universal constant $C >0$) that 
$$
\ABS{ \p ( X > k) - \p (Z > k) } +   \ABS{ \p ( X =  k) - \p (Z =  k) } \leq C \frac{ \pma k^{\f 5 2}}{d}   \p ( Y = k)\,. 
$$
Then, from Lemma \ref{le:lind}, we deduce that it is sufficient to check that for $k \geq d$, 
\begin{equation}\label{eq:finta}
\frac{d \me^{ - d h (( \frac{ k- 2 - d}{ d} )_+)  }}{\p (Y = k)} \leq C \frac{  k^{\f 5 2}}{d} \,. 
\end{equation}
First, from Stirling's formula, we find 
$$
\p ( Y = k ) = \f{d^{k}}{k!}\me^{-d} \geq \frac{  \me^{- k \log (k / d) + k - d} }{ C \sqrt k}= \frac{\me^{-d h ( \frac{ k - d}{ d})}}{ C \sqrt k}   \,. 
$$
Secondly, from the convexity of $ x \mapsto h(x_+)$, we find 
$$
d h \PAR{\PAR{ \frac{ k- 2 - d}{ d} }_+}  = d h \PAR{\PAR{ \frac{ k - d}{ d}  - \f 2 d  }_+} \geq d h \PAR{ \frac{ k - d}{ d}}  - 2 \log \PAR{ \f k d }\,. 
$$
It follows that 
$$
\frac{\me^{ - d h ((\frac{ k- 2 - d}{ d})_+)  }}{\p (Y = k)}  \leq C \sqrt k \PAR{ \f k d }^2\,. 
$$
This concludes the proof of \eqref{eq:finta}.
\end{proof}

\appendix

\section{Auxiliary results}
  
   \beg{lem}\la{lem:covBernoulli}Let, in a \pro space, $E\subset E'$ and $F\subset F'$ be some events. Assume that $(E'\backslash E)  \cap F  = (F'\backslash F)  \cap E   = \emptyset$. Then $$
 \Cov(\one_{E'},\one_{F'})   =  \Cov(\one_E,\one_F) + \p(E'\bck E) \p(F)  +  \p(F'\bck F) \p ( E)\,.$$ \en{lem}
   
   \bpr Set $E'' \deq E'\bck E$ and $F'' \deq F'\bck F$. We have
   \begin{align*}
   &\mspace{-30mu} \Cov(\one_{E'},\one_{F'})-\Cov(\one_E,\one_F)
   \\
   &=\p(E'\cap F')-\p(E\cap F)-(\p(E')\p(F')-\p(E)\p(F))\\
   &=\p((E'\cap F')\bck (E\cap F))-[(\p(E)+\p(E''))(\p(F)+\p(F''))-\p(E)\p(F)]
  \end{align*}
   As $$\ds (E'\cap F')\bck (E\cap F)=(E\cap F'')\cup(E''\cap F)\cup(E''\cap F'')\,,$$ we conclude easily. 
   \epr

  \beg{lem}\la{StarGraph}The adjacency matrix $A$ of a star graph with central degree $D\ge 1$ (see Definition \ref{def:star_graph}) is a $(D+1)\times (D+1)$ real symmetric matrix with nonzero eigenvalues    $\pm\sqrt{D}$ and associated   eigenvectors  $ (\pm\sqrt{D},1, \ld, 1)$  (the first coordinate corresponds to the centre of the star).\end{lem}

\bpr One notices that the matrix   has 
  rank $2$ and that the vectors given here are actually eigenvectors   for $\pm\sqrt{D}$. 
\epr

The following result is \cite[Theorem 2.1]{LeVershynin}. It concerns general inhomogeneous Erd\H{o}s-R\'enyi graphs with mean adjacency matrix $(p_{ij})_{i,j \in [n]}$. We recall that a weighted graph has adjacency matrix $A'$ whose entries are nonnegative real numbers, with the entry $A'_{ij} \geq 0$ denoting the weight of the edge $\{i,j\}$.

 \beg{Th}\la{LeVer} Set $\pma \deq \max_{i,j}p_{ij}$ and choose $r\ge 1$. Then the following holds with \pro at least $1-n^{-r}$. Consider a subset of at most $10\pma^{-1}$ vertices and reduce  the weights of the edges incident to those vertices in an arbitrary way. Then the adjacency matrix $A'$ of the new (weighted) graph satisfies $$\|A'-\E A\|\le Cr^{3/2}(\sqrt{n\pma}+\sqrt{d'})\,,$$ where $C$ is a constant independent of $r$, $d'=\max_{1\le i\le n}\|R'_i\|_{\ell^1} $ with $R_1'$, \ld, $R_n'$ the rows of $A'$. 
 \en{Th}

\bibliographystyle{abbrv}
\bibliography{bib}

\bigskip

\noindent
Florent Benaych-Georges, 
MAP 5 (CNRS UMR 8145) - Universit\'e Paris Descartes, 45 rue des Saints-P\`eres 75270 Paris cedex~6,  France. Email:
\href{mailto:florent.benaych-georges@parisdescartes.fr}{florent.benaych-georges@parisdescartes.fr}.
\\[1em]
Charles Bordenave, 
Institut de Math\'ematiques (CNRS UMR 5219) - Universit\'e
Paul Sabatier. 31062 Toulouse cedex 09. France. Email:
\href{mailto:bordenave@math.univ-toulouse.fr}{bordenave@math.univ-toulouse.fr}.
\\[1em]
Antti Knowles, 
University of Geneva, Section of Mathematics, 2-4 Rue du Li\`evre, 1211 Gen\`eve 4, Switzerland. 
Email:
\href{mailto:antti.knowles@unige.ch}{antti.knowles@unige.ch}.

\bigskip

\paragraph{Acknowledgements}
C.\ B.\ is supported by grants ANR-14-CE25-0014 and ANR-16-CE40-0024-01. A.\ K.\ is supported by the Swiss National Science Foundation and the European Research Council.

\end{document}